\documentclass[11pt]{amsart}
\usepackage{amsmath}
\usepackage{amssymb}
\usepackage{bm}
\usepackage{ esint }
\usepackage{graphicx}
\usepackage{verbatim}
\usepackage{enumitem}
\setlist[enumerate,1]{leftmargin=*}
\setlist[itemize,1]{leftmargin=*}
\usepackage{mathrsfs}
\usepackage[dvipsnames]{color}
\usepackage{todonotes}
\usepackage[pdftex,pdfstartview=FitH,pdfborderstyle={/S/B/W 1},%
colorlinks=true, linkcolor=blue, urlcolor=blue, citecolor=blue,%
pagebackref=true]{hyperref}
\usepackage[nobysame,alphabetic,initials,msc-links]{amsrefs}
\usepackage{comment}
\usepackage{mathtools}

\calclayout 

\newtheorem{lemma}{Lemma}[section]
\newtheorem{theorem}[lemma]{Theorem}
\newtheorem{proposition}[lemma]{Proposition}

\theoremstyle{definition}

\newtheorem{definition}[lemma]{Definition}

\newtheorem{example}[lemma]{Example}

\numberwithin{equation}{section}

\def\geq{\geqslant}
\def\leq{\leqslant}

\def\0{\mathbf{0}}

\def\epsilon{\varepsilon}

\newcommand{\bea}{\begin{eqnarray}}
  \newcommand{\eea}{\end{eqnarray}}
  \newcommand{\beab}{\begin{eqnarray*}}
  \newcommand{\eeab}{\end{eqnarray*}}

  \newcommand{\be}{\begin{equation}}
  \newcommand{\ee}{\end{equation}}

\begin{document}

\title{Slow Entropy of Quasi-Unipotent Actions on
\(S\)-Arithmetic Homogeneous Spaces}

\author{Jasurbek Imomov}

\address{J.I.: Department of Mathematics, National University of Singapore,
119076, Singapore}

\medskip

 \email{e1445244@u.nus.edu}

\date{}

\begin{abstract}
We study metric and topological slow entropy for translations on a
class of finite volume $S$-arithmetic homogeneous spaces with real
and non-Archimedean factors. We consider translations whose real
component is quasi-unipotent and whose non-Archimedean component
lies in a compact open subgroup. At polynomial scales, we show that
the metric and topological slow entropies coincide and are determined
by the chain structure of the real adjoint action. In particular,
the non Archimedean components contribute no additional polynomial
slow entropy. Our result extends the slow entropy formula of
Kanigowski, Vinhage, and Wei \cite{KVW} for quasi-unipotent flows
on real homogeneous spaces to this $S$-arithmetic setting.
\end{abstract}

\maketitle
\section{Introduction}

Many dynamical systems of interest have zero classical (Kolmogorov--Sinai) entropy while
still exhibiting nontrivial orbit complexity. This is particularly
common in parabolic dynamics, where nearby trajectories may diverge at
polynomial rather than exponential rates. Since classical entropy
records complexity at the exponential scale, it does not distinguish
such slower growth. This motivates the study of slow entropy, which
refines classical entropy by measuring orbit complexity at
subexponential scales.

Katok and Thouvenot \cite{KT} developed slow entropy type invariants
for measure preserving actions of amenable groups by considering
coverings of orbit names by Hamming balls. Their approach allows the
asymptotic growth of these coverings to be measured with respect to a
chosen family of scales. Thus, instead of the exponential family
$e^{\chi N}$ associated with classical entropy, one may consider slower
families such as the polynomial scale $N^\chi$, which can yield
nontrivial invariants for systems with zero classical entropy. They
also used slow entropy to obtain obstructions to smooth realization of
measure preserving actions.

Ferenczi \cite{Ferenczi} developed the related notion of
measure theoretic complexity and showed that its uniform boundedness
(equivalently, vanishing slow entropy at every scale) characterizes
Kronecker systems.

Slow entropy for higher rank actions has also been studied
in \cite{KKR} and \cite{Ho}. Kanigowski \cite{Kanigowski}
computed slow entropy for certain smooth surface flows. Adams \cite{Adams21Gene} established genericity results
and proved the existence of rigid, weakly mixing
transformations with infinite lower slow entropy
for each subexponential rate.
More recently, Banerjee, Kunde, and Wei \cites{BKW,BKW2}
established flexibility results for polynomial slow entropy
of certain measure preserving transformations, while Cheng,
Ospina, Vinhage, and Zhai \cite{COVZ} investigated variational
properties of slow entropy, including the validity and failure
of the variational principle.

In the homogeneous setting, Kanigowski, Vinhage, and Wei \cite{KVW}
computed the polynomial slow entropy of quasi-unipotent flows on
finite volume homogeneous spaces, giving an explicit formula in terms
of the Jordan chain structure of the adjoint action and showing that
the metric and topological slow entropies coincide. This was
subsequently extended by Kanigowski, Kunde, Vinhage, and Wei
\cite{KKVW} to higher rank abelian unipotent actions. They expressed
the polynomial slow entropy in terms of the dimensions of a
decomposition of the ambient Lie algebra induced by the Lie algebra of
the acting unipotent group, and again proved that the metric and
topological slow entropies coincide.

The results of \cite{KVW} and \cite{KKVW} concern actions on real homogeneous
spaces. It is therefore natural to ask how the picture changes in
the $S$-arithmetic setting, where real and non-Archimedean components
occur simultaneously. 
We address this question for translations
on $S$-arithmetic homogeneous spaces.

\subsection{Slow entropy invariants}

\subsubsection{Topological slow entropy}

Let $(X,d)$ be a metric space, and let $T:X\to X$ be a uniformly
continuous transformation. For $N\geq 1$, define
\[
d_T^N(x,y):=\max_{0\leq n<N}d(T^nx,T^ny),
\]
and the Bowen ball
\[
B_T^N(x,\varepsilon)
:=\{y\in X:d_T^N(x,y)<\varepsilon\}.
\]
For a compact set $K\subset X$, let $N_{T,K}(\varepsilon,N)$
denote the minimal number of Bowen balls $B_T^N(x,\varepsilon)$
required to cover $K$, and let $S_{T,K}(\varepsilon,N)$
denote the maximal number of pairwise disjoint Bowen balls
$B_T^N(x,\varepsilon)$ with centers in $K$.

For every compact $K\subset X$, $\varepsilon>0$, and $N\geq 1$,
\[
N_{T,K}(2\varepsilon,N)
\leq S_{T,K}(\varepsilon,N)
\leq N_{T,K}\left(\frac{\varepsilon}{2},N\right).
\]

\begin{definition}[Topological slow entropy]
The topological slow entropy of $T$ with respect to the polynomial
scale $a_\chi(N)=N^\chi$ is defined by
\[
h_{\mathrm{top},a_\chi}(T)
:=
\sup_{\substack{K\subset X\\ K\text{ compact}}}
\lim_{\varepsilon\to 0}\left(\sup\left\{
\chi:
\limsup_{N\to\infty}
\frac{N_{T,K}(\varepsilon,N)}{N^\chi}>0
\right\}\right).
\]
Equivalently, $N_{T,K}(\varepsilon,N)$ may be replaced by
$S_{T,K}(\varepsilon,N)$ in the definition.
\end{definition}

\subsubsection{Metric slow entropy}

Let $(X,\mathcal B,\mu,T)$ be a measure preserving system and let
$\xi$ be a finite measurable partition of $X$. For $N\geq 1$, define
the normalized Hamming distance
\[
d_\xi^N(x,y)
:=
\frac{1}{N}
\left|
\left\{
0\leq n<N:
T^n x \text{ and } T^n y
\text{ lie in different atoms of }\xi
\right\}
\right|.
\]

Let $S_\xi(T,N,\varepsilon)$ denote the minimal number of
$d_\xi^N$-balls of radius $\varepsilon$ whose union has measure
greater than $1-\varepsilon$.
\begin{definition}[Metric slow entropy]
The metric slow entropy of $T$ with respect to the polynomial
scale $a_\chi(N)=N^\chi$ is defined by
\[
h_{\mu,a_\chi}(T)
:=
\sup_{\xi}
\lim_{\varepsilon\to0}
\left(\sup\left\{
\chi:
\limsup_{N\to\infty}
\frac{S_\xi(T,N,\varepsilon)}{N^\chi}>0
\right\}\right),
\]
where the supremum is taken over all finite measurable partitions
$\xi$ of $X$.
\end{definition}

\subsection{Statement of the main theorem.}
Recall that for a connected Lie group $G$ with Lie algebra
$\mathfrak g:=\operatorname{Lie}(G)$, an element $U\in\mathfrak g$
is called $\operatorname{ad}$-\emph{quasi-unipotent} if every
eigenvalue of $\operatorname{ad}U$ is purely imaginary.
In this case, $u:=\exp(U)$ is called \emph{quasi-unipotent}.
Then, the real Jordan
normal form of $\operatorname{ad}U$ is
\[
\operatorname{diag}\left(
J_{m_1},\ldots,J_{m_l},
J_{m_{l+1}}^{\alpha_1},\ldots,
J_{m_n}^{\alpha_{n-l}}
\right),
\]
where $J_m$ is the $m\times m$ nilpotent Jordan block, and $J_m^\alpha$ is the $2m\times 2m$ matrix given by
\[
J_m^\alpha=
\begin{pmatrix}
Q_\alpha&I_2&&\\
&Q_\alpha&\ddots&\\
&&\ddots&I_2\\
&&&Q_\alpha
\end{pmatrix},
\qquad
Q_\alpha=
\begin{pmatrix}
0&\alpha\\
-\alpha&0
\end{pmatrix},
\quad \alpha\ne0.
\]

The tuple
\[
(m_1,\ldots,m_l;m_{l+1},\ldots,m_n)
\]
is called the \emph{chain structure} of $U$.

Let $\mathbf G\leq\mathrm{SL}_d$ be an algebraic group defined over
$\mathbb Q$ such that $\mathbf G(\mathbb R)$ is connected.
Let $S$ be a finite set consisting of $\infty$ and finitely
many primes. Define $\mathbb Q_\infty:=\mathbb R$ and
\[
\mathbb Q_S:=\prod_{p\in S}\mathbb Q_p,
\qquad
\mathbf G(\mathbb Q_S):=
\prod_{p\in S}\mathbf G(\mathbb Q_p),
\]
\[
\mathbb Z_S:=
\mathbb Z\left[\frac1p:p\in S\setminus\{\infty\}\right],
\qquad
\Gamma_S:=\Delta\bigl(\mathbf G(\mathbb Z_S)\bigr),
\]
where $\Delta(\gamma):=(\gamma)_{p\in S}$.
We assume that $\Gamma_S$ is a lattice in
$\mathbf G(\mathbb Q_S)$ and that
$\Delta_f(\mathbf G(\mathbb Z_S))$ is dense in
\[\prod_{p\in S\setminus\{\infty\}}\mathbf G(\mathbb Q_p)\]
with respect to the product topology, where
$\Delta_f(\gamma):=(\gamma)_{p\in S\setminus\{\infty\}}$.

The lattice assumption is equivalent to
$\mathbf G(\mathbb Z)$ being a lattice in
$\mathbf G(\mathbb R)$; see \cite{Gui}*{Theorem~2.6}.
If $\mathbf G$ is Zariski connected, then the
Borel--Harish-Chandra theorem implies that
$\mathbf G(\mathbb Z)$ is a lattice in
$\mathbf G(\mathbb R)$ if and only if
$\mathbf G$ has no nontrivial $\mathbb Q$-rational
algebraic characters.
In particular, this holds whenever
$\mathbf G$ is semisimple.

The density assumption follows from strong approximation
whenever $\mathbf G$ is semisimple and simply connected
and every $\mathbb Q$-simple factor has noncompact real
points; see \cite{PR}*{Theorem~7.12}.
In particular, both assumptions hold for
$\mathbf G=\mathrm{SL}_d$, $d\geq2$.

\medskip

Consider the $S$-arithmetic homogeneous space
\[
X_S:=
\left(\prod_{p\in S}\mathbf G(\mathbb Q_p)\right)
\big/\Delta\bigl(\mathbf G(\mathbb Z_S)\bigr).
\]
We denote by $\mu_S$ the normalized
$\mathbf G(\mathbb Q_S)$-invariant probability measure
on $X_S$. Let $u\in\mathbf G(\mathbb R)$ be
quasi-unipotent with chain structure
\[
(m_1,\ldots,m_l;m_{l+1},\ldots,m_n).
\]
For
\[
v\in\prod_{p\in S\setminus\{\infty\}}
\mathbf G(\mathbb Z_p),
\]
consider the translation induced by left multiplication
\[
T_S:X_S\longrightarrow X_S,\qquad
T_S(x):=(u,v)x.
\]
Define
\[
R:=
\sum_{i=1}^{l}\binom{m_i}{2}
+2\sum_{i=l+1}^{n}\binom{m_i}{2}.
\]

Our main result is the following.

\begin{theorem}\label{thm:main}
For the polynomial scale $a_\chi(N)=N^\chi$,
\[
h_{\mu_S,a_\chi}(T_S)
=
h_{\mathrm{top},a_\chi}(T_S)
=
R.
\]
\end{theorem}

\medskip

We briefly describe the proof. We begin by identifying the
$S$-arithmetic homogeneous space with a more convenient homogeneous space
model. This identification is equivariant, uniform, and measure preserving,
so both metric slow entropy and topological slow entropy
remain unchanged. We may therefore work with this model throughout the proof.

The next step is to project onto real homogeneous space. This projection
gives a measure preserving factor of the dynamical system. The entropy formula
of \cite{KVW}*{Theorem~1.11}, together with monotonicity of slow entropy under
factors, then gives the desired lower bound $R$ for the metric slow entropy. A slow version of Goodwyn's inequality further
bounds the metric slow entropy by topological
slow entropy.

Thus, the main remaining task is to obtain the reverse upper bound for the topological slow entropy. We prove this by estimating the volumes of Bowen
balls. Combining this estimate with the preceding lower bounds yields the
desired equality.

\par\medskip
\noindent\textbf{Acknowledgements:}
The author would like to thank his advisor, Daren Wei, for his guidance, encouragement, and many helpful discussions. The author is also grateful to Huu Hoan for his helpful feedback on the manuscript.
\par\medskip

\medskip
\noindent\textbf{AI Disclosure.}
ChatGPT (OpenAI) was used to assist with English language
editing, including grammatical corrections and polishing
the English in the Introduction, including descriptions
of previous research, as well as LaTeX formatting.
The main mathematical ideas, results, and proofs
are the author's own.

\section{Preliminaries}

\subsection{Homogeneous spaces and metrics considered}

Let $L$ be a locally compact second countable group and $\Lambda<L$ a
lattice (that is, a discrete subgroup such that $L/\Lambda$ carries a
finite $L$-invariant measure). Let $m_L$ be a Haar measure on $L$, and
let $\mu_{L/\Lambda}$ be the induced $L$-invariant measure on
$L/\Lambda$. We normalize $m_L$ so that $\mu_{L/\Lambda}$ is a
probability measure.

Let $d_L$ be a right invariant metric on $L$ compatible with its
topology. The induced quotient metric on $L/\Lambda$ is
\[
d_{L/\Lambda}(g_1\Lambda,g_2\Lambda)
:=
\inf_{\gamma\in\Lambda} d_L(g_1\gamma,g_2).
\]
Let $q:L\to L/\Lambda$ denote the natural projection map. For a compact set
$K\subset L/\Lambda$, define its injectivity radius by
\[
\operatorname{inj}(K)
:=
\sup\left\{
r>0:
q|_{B_{d_L}(g,r)}
\text{ is injective for every }g\Lambda\in K
\right\}.
\]
Note that every compact set $K\subset L/\Lambda$ has positive
injectivity radius:
\[
\operatorname{inj}(K)>0.
\]
Moreover, if $A\subset L$ is measurable and $q$ is injective on $A$,
then
\[
\mu_{L/\Lambda}(q(A))=m_L(A).
\]

\medskip

\noindent\textbf{Paper-specific notation.}
We fix the following notation for the remainder of the paper.
Define
\[
\begin{gathered}
G:=\mathbf G(\mathbb R),\qquad
G_p:=\mathbf G(\mathbb Q_p),\qquad
K_p:=\mathbf G(\mathbb Z_p),\\
G_f:=\prod_{p\in S\setminus\{\infty\}}\mathbf G(\mathbb Q_p),\qquad
K_f:=\prod_{p\in S\setminus\{\infty\}}\mathbf G(\mathbb Z_p),\\
G_S:=\prod_{p\in S}\mathbf G(\mathbb Q_p),\qquad
\Gamma_S:=\Delta\bigl(\mathbf G(\mathbb Z_S)\bigr),
\qquad X_S:=G_S/\Gamma_S,\\
H:=G\times K_f,\qquad
\Gamma:=\Delta\bigl(\mathbf G(\mathbb Z)\bigr),\\
X:=H/\Gamma,\qquad
Y:=G/\mathbf G(\mathbb Z).
\end{gathered}
\]

We denote by $T_S$, $T_X$, and $T_Y$ the transformations induced by
left multiplication by $(u,v)$, $(u,v)$, and $u$ on $X_S$, $X$, and $Y$,
respectively, where $v\in K_f$.

Since topological slow entropy is independent of the choice of
compatible right invariant metric, we fix a right invariant
Riemannian metric $d_G$ on $G$ and, for each
$p\in S\setminus\{\infty\}$, a compatible proper
right invariant metric $d_p$ on $G_p$ whose restriction to $K_p$
is bi-invariant and given by
\[
d_p(g,h)=\|g-h\|_p
:=\max_{i,j}|g_{ij}-h_{ij}|_p;
\]
see \cite{Ruhr}. Define
\[
d_f((g_p)_p,(h_p)_p)
:=
\max_{p\in S\setminus\{\infty\}}d_p(g_p,h_p).
\]
Equip $G_S$ and $H$ with the metrics
\[
d_{G_S}:=\max\{d_G,d_f\},
\qquad
d_H:=d_{G_S}|_{H\times H},
\]
and denote the induced quotient metrics on $X_S,X,Y$ by
$d_S,d_X,d_Y$, respectively.

\section{Proof of the Main Theorem}\label{sec:proof}

\subsection{Preliminary results}

We begin by passing from the $S$-arithmetic quotient $X_S$ to the
quotient $X=H/\Gamma$.

\begin{proposition}\label{prop:uniform-conjugacy}
The natural map
\[
\Phi:X\longrightarrow X_S,\qquad
\Phi(h\Gamma)=h\Gamma_S,
\]
satisfies the following:
\begin{enumerate}[label=(\roman*)]
    \item $\Phi$ is a uniform homeomorphism satisfying
\[
\Phi(hx)=h\Phi(x),\qquad h\in H,\ x\in X.
\]

    \item $\Gamma$ is a lattice in $H$, and, if $\mu_X$ denotes the
    normalized $H$-invariant probability measure on $X$, then
    \[
    \Phi_*\mu_X=\mu_S;
    \]

    \item $\Phi$ conjugates $T_X$ and $T_S$. Consequently,
    \[
    h_{\mu_X,a_\chi}(T_X)=h_{\mu_S,a_\chi}(T_S),
    \qquad
    h_{\mathrm{top},a_\chi}(T_X)
    =
    h_{\mathrm{top},a_\chi}(T_S).
    \]
\end{enumerate}
\end{proposition}

\begin{proof}
We first prove (i). Since
\[
\begin{aligned}
\Gamma_S\cap H
&=\Delta\left(
\left\{\gamma\in\mathbf G(\mathbb Z_S):
\Delta_f(\gamma)\in K_f\right\}
\right)\\
&=\Delta(\mathbf G(\mathbb Z))
=\Gamma.
\end{aligned}
\]

the map $\Phi$ is well defined and injective.

Since $\Delta_f(\mathbf G(\mathbb Z_S))$ is dense and $K_f$ is open in $G_f$, we have
\[
G_f=K_f\,\Delta_f(\mathbf G(\mathbb Z_S)),
\]
and therefore $G_S=H\Gamma_S$. This implies that $\Phi$ is surjective.
The $H$-equivariance of $\Phi$ follows directly from its definition.

For $h_1,h_2\in H$, since $\Gamma\subset\Gamma_S$,
\[
\begin{aligned}
d_S\bigl(\Phi(h_1\Gamma),\Phi(h_2\Gamma)\bigr)
&=\inf_{\eta\in\Gamma_S}d_{G_S}(h_1\eta,h_2)\\
&\leq \inf_{\eta\in\Gamma}d_{G_S}(h_1\eta,h_2)
=d_X(h_1\Gamma,h_2\Gamma).
\end{aligned}
\]
Thus $\Phi$ is uniformly continuous.

To prove uniform continuity of $\Phi^{-1}$, choose $\rho>0$
such that
\[
B_{d_f}(e,\rho)\subset K_f,
\]
which is possible since $K_f$ is open in $G_f$.
Given $\varepsilon>0$, let
$0<\delta<\min\{\varepsilon,\rho\}$. Suppose
\[
d_S\bigl(\Phi(h_1\Gamma),\Phi(h_2\Gamma)\bigr)<\delta,
\qquad h_1,h_2\in H.
\]
Then there exists $\gamma\in\mathbf G(\mathbb Z_S)$ such that
\[
d_{G_S}\bigl(h_1\Delta(\gamma),h_2\bigr)<\delta.
\]
Write $h_i=(g_i,k_i)$ with $k_i\in K_f$.
This follows that
\[
d_f(k_1\Delta_f(\gamma),k_2)<\delta.
\]
By right invariance of $d_f$,
\[
d_f(k_1\Delta_f(\gamma)k_2^{-1},e)<\rho.
\]
Hence $k_1\Delta_f(\gamma)k_2^{-1}\in K_f$,
so $\Delta_f(\gamma)\in K_f$. Therefore,
\[
\gamma\in\mathbf G(\mathbb Z_S)
\cap\bigcap_{p\in S\setminus\{\infty\}}K_p
=\mathbf G(\mathbb Z).
\]
Thus $\Delta(\gamma)\in\Gamma$, and consequently
\[
d_X(h_1\Gamma,h_2\Gamma)\leq d_{H}\bigl(h_1\Delta(\gamma),h_2\bigr)
= d_{G_S}\bigl(h_1\Delta(\gamma),h_2\bigr)
<\delta<\varepsilon.
\]
This proves (i).

\medskip

For (ii), pull back $\mu_S$ by $\Phi$:
\[
\nu:=(\Phi^{-1})_*\mu_S.
\]
Since $\Phi$ is $H$-equivariant, $\nu$ is an $H$-invariant
probability measure on $H/\Gamma$. Hence $\Gamma$ is a lattice in
$H$, and by uniqueness of the normalized $H$-invariant measure,
$\nu=\mu_X$. Equivalently,
\[
\Phi_*\mu_X=\mu_S.
\]

\medskip

For (iii), by $H$-equivariance,
\[
\Phi\circ T_X=T_S\circ\Phi.
\]
Thus $\Phi$ is both a measure-preserving conjugacy and a uniform
conjugacy. The invariance properties of slow entropy therefore give
\[
h_{\mu_X,a_\chi}(T_X)=h_{\mu_S,a_\chi}(T_S),
\qquad
h_{\mathrm{top},a_\chi}(T_X)
=
h_{\mathrm{top},a_\chi}(T_S),
\]
which proves (iii).
\end{proof}

\medskip

Next, we prove one side of the desired entropy equality using
the real homogeneous space $Y$.
\begin{proposition}\label{prop:factor}

Define
\[
\pi:X\longrightarrow Y,\qquad
\pi((g,k)\Gamma)=g\mathbf G(\mathbb Z),
\]
where $g\in G$ and $k\in K_f$, and let
$\mu_Y:=\pi_*\mu_X.$
Then $\mu_Y$ is a $G$-invariant
probability measure on $Y$, and hence $\mathbf G(\mathbb Z)$ is a
lattice in $G$. Moreover, $\pi$ is a factor map from
$(X,\mu_X,T_X)$ onto $(Y,\mu_Y,T_Y)$. Consequently,
\[
h_{\mu_X,a_\chi}(T_X)
\geq
h_{\mu_Y,a_\chi}(T_Y).
\]
\end{proposition}
\begin{proof}
The map $\pi$ is well defined and surjective. Moreover, it is
$G$-equivariant, where $G$ acts on $X$ through $G\times\{e\}\subset H$.
Since $\mu_X$ is $H$-invariant, $\mu_Y=\pi_*\mu_X$ is therefore a
$G$-invariant probability measure on $Y$. Hence
$\mathbf G(\mathbb Z)$ is a lattice in $G$.

For every $(g,k)\Gamma\in X$,
\[
\pi\bigl(T_X((g,k)\Gamma)\bigr)
=\pi((ug,vk)\Gamma)
=ug\mathbf G(\mathbb Z)
=T_Y\bigl(\pi((g,k)\Gamma)\bigr).
\]
Thus $\pi\circ T_X=T_Y\circ\pi$, so $\pi$ is a measure preserving
factor map. The factor monotonicity of metric slow entropy gives
\[
h_{\mu_X,a_\chi}(T_X)
\geq h_{\mu_Y,a_\chi}(T_Y).
\]
\end{proof}

Note that \cite{KT}*{Proposition~2} also applies to $(X,\mu_X,T_X)$:
\begin{theorem}[Slow Goodwyn's theorem, \cite{KT}*{Proposition~2}] \label{thm:slow-goodwyn}
We have
\[
h_{\mu_X,a_\chi}(T_X)
\leq
h_{\mathrm{top},a_\chi}(T_X).
\]
\end{theorem}

We next prove that the topological slow entropy is bounded above by R.
We begin with a volume estimate in the Lie algebra. Choose a chain basis of $\mathfrak g$ corresponding to the chain structure of
$\operatorname{ad}U$. For the
real chains, write
\[
e_{i,0},e_{i,1},\ldots,e_{i,m_i-1},
\qquad 1\leq i\leq l,
\]
and for the rotation chains,
\[
e^1_{i,0},e^2_{i,0},\ldots,
e^1_{i,m_i-1},e^2_{i,m_i-1},
\qquad l+1\leq i\leq n.
\]
Thus every $V\in\mathfrak g$ can be written uniquely as
\[
V
=
\sum_{i=1}^{l}\sum_{j=0}^{m_i-1}x_{i,j}e_{i,j}
+
\sum_{i=l+1}^{n}\sum_{j=0}^{m_i-1}
\bigl(x_{i,j}e^1_{i,j}+y_{i,j}e^2_{i,j}\bigr).
\]
Define
\[
\|V\|_\infty
=
\max\left\{
|x_{i,j}|,
\sqrt{x_{r,s}^{\,2}+y_{r,s}^{\,2}}
\right\},
\]
where
\[
1\leq i\leq l,\quad 0\leq j<m_i,
\qquad
l<r\leq n,\quad 0\leq s<m_r.
\]
For $\varepsilon>0$ and $N\geq1$, set
\[
B_N(0,\varepsilon)
=
\left\{
V\in\mathfrak g:
\|\operatorname{Ad}(u^n)V\|_\infty<\varepsilon
\text{ for all }0\leq n<N
\right\}.
\]
Let $m$, $m_f$, and $\lambda$ be Haar measures on
$G$ and $K_f$, and Lebesgue measure on $\mathfrak g$,
respectively.

\begin{proposition}[\cite{KVW}, Corollary~2.8]
\label{prop:lie-algebra-volume}
For each $\varepsilon>0$, there exists
$c(\varepsilon)>0$, independent of $N$, such that
\[
\lambda\bigl(B_N(0,\varepsilon)\bigr)
\geq c(\varepsilon)N^{-R},
\qquad N\geq1.
\]
\end{proposition}
The following proposition provides a local comparison between Lebesgue
measure on $\mathfrak g$ and Haar measure on $G$.
\begin{proposition}\label{prop:exp-volume-comparison}
There exist $\delta_0>0$ and $c_0>0$ such that, for every
$0<\delta<\delta_0$ and $N\geq1$,
\[
m\bigl(\exp(B_N(0,\delta))\bigr)
\geq
c_0\,\lambda\bigl(B_N(0,\delta)\bigr).
\]
\end{proposition}
\begin{proof}
Choose $\delta_0>0$ such that
\[
\exp:B_{\|\cdot\|_\infty}(0,\delta_0)\longrightarrow G
\]
is a diffeomorphism onto its image. Since m is Haar measure on G, there exists a positive continuous
function $J$ on $B_{\|\cdot\|_\infty}(0,\delta_0)$ such that
\[
m(\exp A)=\int_A J(V)\,d\lambda(V)
\]
for every measurable $A\subset B_{\|\cdot\|_\infty}(0,\delta_0)$.
Since $J(0)>0$, after decreasing $\delta_0$ if necessary, there exists
$c_0>0$ such that
\[
J(V)\geq c_0
\]
for all $V\in B_{\|\cdot\|_\infty}(0,\delta_0)$.

Since
\[
B_N(0,\delta)
\subset
B_{\|\cdot\|_\infty}(0,\delta)
\subset
B_{\|\cdot\|_\infty}(0,\delta_0),
\]
we obtain
\[
m\bigl(\exp(B_N(0,\delta))\bigr)
=
\int_{B_N(0,\delta)}J(V)\,d\lambda(V)
\geq
c_0\,\lambda\bigl(B_N(0,\delta)\bigr).
\]
\end{proof}

\subsection{Main results}

Let $K\subset X$ be compact. Since
\[
\pi:X\longrightarrow Y,\qquad
\pi((g,k)\Gamma)=g\mathbf G(\mathbb Z)
\]
is continuous, $\pi(K)\subset Y$ is compact. Thus,
\[
\operatorname{inj}(\pi(K))>0.
\]
Choose $\rho>0$ such that
\[
B_{d_f}(e,\rho)\subset K_f,
\]
and let
\[
0<\varepsilon<\min\{\rho,\operatorname{inj}(\pi(K))\}.
\]
By right invariance of $d_f$,
\[
B_{d_f}(k,\varepsilon)\subset K_f
\qquad\text{for every }k\in K_f.
\]
Then choose $\delta>0$ sufficiently small so that $\exp$ is a
diffeomorphism on $B_{\|\cdot\|_\infty}(0,\delta)$ and
\[
\|V\|_\infty<\delta
\quad\Longrightarrow\quad
d_G(\exp V,e)<\varepsilon.
\]
For $x=(g,k)\Gamma\in K$, define
\[
\Omega_x
=
\exp(B_N(0,\delta))g
\times
B_{d_f}(k,\varepsilon),
\]
and let $q:H\to X=H/\Gamma$ be the natural projection
given by $q(h)=h\Gamma$.


\begin{proposition}\label{prop:bowen-neighborhood}
For every $x\in K$, the restriction $q|_{\Omega_x}$ is injective and
\[
q(\Omega_x)\subset B_{T_X}^N(x,\varepsilon).
\]
\end{proposition}
\begin{proof}
We first prove that $q|_{\Omega_x}$ is injective. Suppose
\[
q(\exp(V)g,k_1)=q(\exp(W)g,k_2),
\]
where $V,W\in B_N(0,\delta)$ and
$k_1,k_2\in B_{d_f}(k,\varepsilon)$. Since
$\Gamma=\Delta(\mathbf G(\mathbb Z))$, there exists
$\gamma\in\mathbf G(\mathbb Z)$ such that
\[
(\exp(V)g,k_1)
=
(\exp(W)g,k_2)\Delta(\gamma).
\]
In particular,
\[
\exp(V)g=\exp(W)g\gamma.
\]

Since $V,W\in B_N(0,\delta)$, we have
\[
\|V\|_\infty,\ \|W\|_\infty<\delta.
\]
Hence, by the choice of $\delta$ and right invariance of $d_G$,
\[
d_G(\exp(V)g,g)<\varepsilon,
\qquad
d_G(\exp(W)g,g)<\varepsilon.
\]
Thus both $\exp(V)g$ and $\exp(W)g$ lie in
$B_{d_G}(g,\varepsilon)$. Moreover,
\[
\exp(V)g\mathbf G(\mathbb Z)
=
\exp(W)g\gamma\mathbf G(\mathbb Z)
=
\exp(W)g\mathbf G(\mathbb Z).
\]
Since
\[
g\mathbf G(\mathbb Z)=\pi(x)\in\pi(K)
\]
and $\varepsilon<\operatorname{inj}(\pi(K))$, the quotient map
\[
G\longrightarrow Y=G/\mathbf G(\mathbb Z)
\]
is injective on $B_{d_G}(g,\varepsilon)$. Therefore
\[
\exp(V)g=\exp(W)g.
\]
Hence $\exp(V)=\exp(W)$ and $\gamma=e$. Choice of $\delta$ yields $V=W$.
The second coordinate then gives $k_1=k_2$. Thus
$q|_{\Omega_x}$ is injective.

We now prove the inclusion. Let
\[
y=q(\exp(V)g,k_1)\in q(\Omega_x),
\]
where $V\in B_N(0,\delta)$ and
$k_1\in B_{d_f}(k,\varepsilon)$. For $0\leq n<N$,
\[
T_X^n x=(u^ng,v^nk)\Gamma,
\qquad
T_X^n y=(u^n\exp(V)g,v^nk_1)\Gamma.
\]
Using
\[
u^n\exp(V)u^{-n}
=
\exp\bigl(\operatorname{Ad}(u^n)V\bigr),
\]
we obtain
\[
u^n\exp(V)g
=
\exp\bigl(\operatorname{Ad}(u^n)V\bigr)u^ng.
\]
Since $V\in B_N(0,\delta)$, for all $0\leq n<N$,
\[
\|\operatorname{Ad}(u^n)V\|_\infty<\delta.
\]
Therefore, by the choice of $\delta$ and right invariance of $d_G$,
\[
d_G(u^n\exp(V)g,u^ng)
=
d_G\bigl(\exp(\operatorname{Ad}(u^n)V),e\bigr)
<\varepsilon.
\]

Since $v\in K_f$ and $d_f$ is bi-invariant on $K_f$,
\[
d_f(v^nk_1,v^nk)
=
d_f(k_1,k)
<\varepsilon.
\]
Hence, for the metric on $H$,
\[
d_H\bigl(
(u^n\exp(V)g,v^nk_1),
(u^ng,v^nk)
\bigr)
<\varepsilon.
\]
Since $d_X$ is the induced quotient metric, it follows that
\[
d_X(T_X^ny,T_X^nx)<\varepsilon,
\qquad 0\leq n<N.
\]
Thus
\[
y\in B_{T_X}^N(x,\varepsilon).
\]
Since $y\in q(\Omega_x)$ was arbitrary,
\[
q(\Omega_x)\subset B_{T_X}^N(x,\varepsilon).
\]
\end{proof}

We now combine the preceding estimates to complete the upper bound for
topological slow entropy.

\begin{proposition}\label{prop:topological-upper-bound}
Let $K\subset X$ be compact. Choose $\rho>0$ such that
$B_{d_f}(e,\rho)\subset K_f$. For every
\[
0<\varepsilon<
\min\{\rho,\operatorname{inj}(\pi(K))\},
\]
there exists $C(K,\varepsilon)>0$ such that
\[
S_{T_X,K}(\varepsilon,N)
\leq C(K,\varepsilon)N^R,
\qquad N\geq1.
\]
Consequently,
\[
h_{\mathrm{top},a_\chi}(T_X)\leq R.
\]
\end{proposition}

\begin{proof}
Fix
\[
0<\varepsilon<
\min\{\rho,\operatorname{inj}(\pi(K))\},
\]
and choose $\delta=\delta(\varepsilon)>0$ sufficiently small so that
$\exp$ is a diffeomorphism on
$B_{\|\cdot\|_\infty}(0,\delta)$ and
\[
\|V\|_\infty<\delta
\quad\Longrightarrow\quad
d_G(\exp V,e)<\varepsilon.
\]
For $N\geq1$, let
\[
M=S_{T_X,K}(\varepsilon,N).
\]
Choose
\[
x_i=(g_i,k_i)\Gamma\in K,\qquad 1\leq i\leq M,
\]
as centers of pairwise disjoint $(\varepsilon,N)$-Bowen balls. Define
\[
\Omega_{x_i}
=
\exp(B_N(0,\delta))g_i
\times
B_{d_f}(k_i,\varepsilon).
\]
By Proposition~\ref{prop:bowen-neighborhood},
\[
q(\Omega_{x_i})\subset B_{T_X}^N(x_i,\varepsilon),
\]
so the sets $q(\Omega_{x_i})$ are pairwise disjoint. Hence
\[
\sum_{i=1}^{M}\mu_X\bigl(q(\Omega_{x_i})\bigr)\leq1.
\]

Since $q$ is injective on each $\Omega_{x_i}$, the quotient measure
formula gives
\[
\mu_X\bigl(q(\Omega_{x_i})\bigr)
=
c\,m\bigl(\exp(B_N(0,\delta))g_i\bigr)
m_f\bigl(B_{d_f}(k_i,\varepsilon)\bigr),
\]
for some constant $c>0$. Since
$\mathbf G(\mathbb Z)$ is a lattice in $G$, the group $G$ is
unimodular, and hence
\[
m\bigl(\exp(B_N(0,\delta))g_i\bigr)
=
m\bigl(\exp(B_N(0,\delta))\bigr).
\]
Moreover, by the bi-invariance of $d_f$ on $K_f$,
\[
\begin{aligned}
m_f\bigl(B_{d_f}(k_i,\varepsilon)\bigr)
&=m_f\bigl(k_iB_{d_f}(e,\varepsilon)\bigr)\\
&=m_f\bigl(B_{d_f}(e,\varepsilon)\bigr)
=:b(\varepsilon)>0.
\end{aligned}
\]
Therefore, by Propositions~\ref{prop:exp-volume-comparison} and
\ref{prop:lie-algebra-volume},
\[
\begin{aligned}
\mu_X\bigl(q(\Omega_{x_i})\bigr)
&=
c\,b(\varepsilon)\,
m\bigl(\exp(B_N(0,\delta))\bigr)\\
&\geq
c\,b(\varepsilon)c_0\,
\lambda\bigl(B_N(0,\delta)\bigr)\\
&\geq
c\,b(\varepsilon)c_0c(\delta)\,N^{-R}\\
&=:c_1(K,\varepsilon)N^{-R},
\end{aligned}
\]
where $c_1(K,\varepsilon)>0$ is independent of $i$ and $N$.
Consequently,
\[
1
\geq
\sum_{i=1}^{M}\mu_X\bigl(q(\Omega_{x_i})\bigr)
\geq
M c_1(K,\varepsilon)N^{-R},
\]
and therefore
\[
S_{T_X,K}(\varepsilon,N)
=
M
\leq
C(K,\varepsilon)N^R,
\]
where $C(K,\varepsilon)=c_1(K,\varepsilon)^{-1}$ is
independent of $N$.

Hence, for every $\chi>R$,
\[
\limsup_{N\to\infty}
\frac{S_{T_X,K}(\varepsilon,N)}{a_\chi(N)}
\leq
C(K,\varepsilon)
\limsup_{N\to\infty}N^{R-\chi}
=0.
\]
Therefore,
\[
\sup\left\{
\chi:
\limsup_{N\to\infty}
\frac{S_{T_X,K}(\varepsilon,N)}{a_\chi(N)}>0
\right\}
\leq R.
\]
Letting $\varepsilon\to0$ and taking the supremum over all compact
$K\subset X$ gives
\[
h_{\mathrm{top},a_\chi}(T_X)\leq R.
\]
\end{proof}

\begin{proof}[Proof of Theorem \normalfont~\ref{thm:main}]
By \cite{KVW}*{Theorem~1.11},
\[
h_{\mu_Y,a_\chi}(T_Y)=R.
\]
Hence, by Proposition~\ref{prop:factor}, Theorem~\ref{thm:slow-goodwyn}, and
Proposition~\ref{prop:topological-upper-bound},
\[
R
\leq h_{\mu_X,a_\chi}(T_X)
\leq h_{\mathrm{top},a_\chi}(T_X)
\leq R.
\]
Therefore,
\[
h_{\mu_X,a_\chi}(T_X)
=
h_{\mathrm{top},a_\chi}(T_X)
=R.
\]
By Proposition~\ref{prop:uniform-conjugacy},
\[
h_{\mu_S,a_\chi}(T_S)
=
h_{\mu_X,a_\chi}(T_X)
=
h_{\mathrm{top},a_\chi}(T_X)
=
h_{\mathrm{top},a_\chi}(T_S)
=R.
\]
\end{proof}
 \section{Examples}
 We conclude with an explicit example illustrating
Theorem~\ref{thm:main} for a non-reductive algebraic group.
We verify the lattice and density assumptions and compute
the slow entropy from the chain structure of the adjoint action.

\begin{example}
Let $H:=\mathrm{SL}_n\ltimes\mathbb A^n$, $n\geq 2$, where $\mathbb A^n$ denotes the $n$-dimensional affine space,
with multiplication $(M,v)(N,w):=(MN,v+Mw)$.
We identify $H$ with an algebraic subgroup of
$\mathrm{SL}_{n+1}$ by
\[
(M,v)\longmapsto
\begin{pmatrix}
M&v\\
0&1
\end{pmatrix}.
\]
Consider the homogeneous space
\[
X:=
\frac{H(\mathbb R)\times H(\mathbb Q_p)}
{\Delta(H(\mathbb Z[1/p]))}.
\]

We first verify the assumptions of Theorem~\ref{thm:main}.
Since $\mathbb Z[1/p]$ is dense in $\mathbb Q_p$ and
$\mathrm{SL}_n(\mathbb Q_p)$ is generated by the elementary matrices
$I+aE_{ij}$, where $a\in\mathbb Q_p$ and $i\neq j$, it follows that
$\mathrm{SL}_n(\mathbb Z[1/p])$ is dense in $\mathrm{SL}_n(\mathbb Q_p)$.
Thus, $H(\mathbb Z[1/p])$ is dense in $H(\mathbb Q_p).$
Therefore, by Proposition~\ref{prop:uniform-conjugacy}, we obtain
a natural isomorphism
\[
\frac{H(\mathbb R)\times H(\mathbb Z_p)}
{\Delta(H(\mathbb Z))}
\cong
\frac{H(\mathbb R)\times H(\mathbb Q_p)}
{\Delta(H(\mathbb Z[1/p]))}.
\]
Since $H(\mathbb Z)$ is a lattice in $H(\mathbb R)$,
any finite volume fundamental domain $F$ for $H(\mathbb Z)$
gives a finite volume fundamental domain
$F\times H(\mathbb Z_p)$ for $\Delta(H(\mathbb Z))$.
Hence  $\Delta(H(\mathbb Z))$ is a lattice in
$H(\mathbb R)\times H(\mathbb Z_p)$.

Consequently, $\Delta(H(\mathbb Z[1/p]))$ is a lattice
in $H(\mathbb R)\times H(\mathbb Q_p)$.

Now let $A:=I_n+E_{12}$ and
$u:=(A,0)\in H(\mathbb Z)$.
Consider the translation induced by left multiplication

\[
T:X\longrightarrow X,\qquad T(x):=(u,u)x.
\] and let $\mu$ be the normalized Haar measure on $X$.
Since $A=\exp(E_{12})$, the infinitesimal generator
of $u$ is $(E_{12},0)$. The Lie algebra is
$\mathfrak h=\mathfrak{sl}_n(\mathbb R)\ltimes\mathbb R^n$,
and
\[
\operatorname{ad}(E_{12},0)(B,w)
=([E_{12},B],E_{12}w).
\]
Since
\[
[E_{12},E_{ij}]
=\delta_{2i}E_{1j}-\delta_{j1}E_{i2},
\]
the $\mathfrak{sl}_n$-part has chains
\[
E_{21}\mapsto E_{11}-E_{22}\mapsto-2E_{12}\mapsto0,
\]
and
\[
E_{2j}\mapsto E_{1j}\mapsto0,\qquad
E_{i1}\mapsto-E_{i2}\mapsto0,
\quad 3\le i,j\le n.
\]
The $\mathbb R^n$-part contributes the chain
$e_2\mapsto e_1\mapsto0$, with $e_3,\ldots,e_n$
annihilated by $E_{12}$. Thus, there is one chain of length $3$
and $2n-3$ chains of length $2$, while chains of length $1$
contribute nothing. Hence
\[
R=\binom{3}{2}+(2n-3)\binom{2}{2}=2n.
\]
Therefore, by Theorem~\ref{thm:main},
\[
h_{\mu,a_\chi}(T)
=
h_{\mathrm{top},a_\chi}(T)
=
R=2n.
\]
\end{example}

\bibliographystyle{plain} 
\bibliography{refs} 

\end{document}